\documentclass[12pt]{amsart}
\usepackage{euscript}
\usepackage{a4wide}
\usepackage[T2A]{fontenc}
\usepackage[utf8]{inputenc}
\usepackage[english]{babel}
\usepackage{amsfonts}
\usepackage{amssymb, amsthm}
\usepackage{amsmath}
\usepackage{graphicx}
\usepackage{booktabs} 
\usepackage{xcolor}

\definecolor{webgreen}{rgb}{0,.5,0}
\definecolor{webbrown}{rgb}{.6,0,0}
\definecolor{RoyalBlue}{cmyk}{1, 0.50, 0, 0}
\definecolor{black}{cmyk}{60,40,40,100}
\usepackage[colorlinks=true, breaklinks=true, urlcolor=webbrown, linkcolor=RoyalBlue, citecolor=webgreen]{hyperref}

\renewcommand{\epsilon}{\varepsilon} 
\renewcommand{\phi}{\varphi} 

\def\geq{\geqslant}
\def\ge{\geqslant}

\def\le{\leqslant}

\def\ls{\lesssim}
\def\gs{\gtrsim}

\def\ol{\overline}
\def\kratno{\lower.5ex\hbox{$\,\vdots\,$}}

\def\dotline{\smallskip\hbox to \hsize{\dotfill}\medskip}

\def\norm[#1]{\left\| #1 \right\|}

\newcommand{\R}{\mathbb{R}}

\newcommand{\Cm}{\mathbb{C}}

\newcommand{\Meas}{\mathcal{M}}
\newcommand{\PW}{\mathcal{PW}}

\renewcommand{\Im}{\mathop{\mathrm{Im}}\nolimits}

\newcommand{\supp}{\mathop{\mathrm{supp}}\nolimits}
\newcommand{\dist}{\mathop{\mathrm{dist}}\nolimits}

\theoremstyle{plain}

\newtheorem{thm}{Theorem}[section]
\newtheorem{corol}[thm]{Corollary}
\newtheorem{prop}[thm]{Proposition}
\newtheorem{lemma}[thm]{Lemma}

\title[Christoffel functions]{Christoffel functions of measures on the real line with divergent logarithmic integral}
\author{Pavel Gubkin}

\thanks{The work was performed at the Saint Petersburg Leonhard Euler International Mathematical Institute and supported by the Ministry of Science and Higher Education of the Russian Federation (agreement no. 075–15–2025–343). The author is a winner of the BASIS PhD Student competition and would like to thank its sponsors and jury}

\address{
	\begin{flushleft}
		Pavel Gubkin: gubkinpavel@pdmi.ras.ru, gubkin.pv@yandex.ru \\
        \vspace{0.1cm}
		 Saint Petersburg University\\
         7/9 Universitetskaya nab., 199034 St. Petersburg, Russia \\
         \vspace{0.1cm}
        St. Petersburg Department of Steklov Mathematical Institute, Russian Academy of Sciences\\ 
		Fontanka 27, 191023 St. Petersburg, Russia
	\end{flushleft}
}

\begin{document}

\begin{abstract}
    Let $\sigma$ be a Poisson-finite measure on the real line. If its logarithmic integral diverges, then the 
    functions from the classical Hardy space $H^2$ with compactly supported Fourier transform are dense in $L^2(\R, \sigma)$. We give a quantitative version of this result, adapting the ideas from the 2020 paper \cite{Borichev2020} by Borichev, Kononova and Sodin, where a similar question was studied in the setting of polynomial approximation. 
\end{abstract}

\maketitle

\section{Introduction}

Let $\sigma$ be a Borel measure on the real line $\R$ that satisfies
\begin{gather*}
    \int_\R\frac{d\sigma(x)}{1 + x^2} < \infty.
\end{gather*}
Consider the set $\mathcal X$ consisting of the functions $f$ that admit a representation
\begin{gather*}
    f(\lambda)= \int_{r_1}^{r_2}e^{i\lambda x}g(x)\,dx,
\end{gather*}
where $0\le r_1 < r_2$ and $g\in C^1[r_1,r_2]$.
Notice that every $f\in \mathcal{X}$ satisfies $|f(x)|\le \frac{C}{1 + |x|}$ with some constant $C > 0$ hence $\mathcal X\subset L^2(\R, \sigma)$. The classical result, see Theorem 8.1 in \cite{Denisov2006} or Section 4.2 in \cite{Dym2008Book}, states that $\mathcal{X}$ is dense in $L^2(\R, \sigma)$ if and only if 
\begin{gather}
    \label{eq: szego assertion}
    \int_\R\frac{\log \sigma'(x)\,dx}{1 + x^2} = -\infty.
\end{gather}
Here $\sigma'$ is the density of the absolutely continuous part of $\sigma$ with respect to the Lebesgue measure on the real line. Further, for every $r > 0$ we denote by $\PW_{[0,r]}$ the Paley-Wiener space of functions with spectrum in $[0,r]$, i.e., functions of the form
\begin{gather*}
    f(\lambda)
    =
    \int_{0}^{r} e^{i\lambda x}g(x)\,dx,
    \qquad
    g\in L^2([0,r]).
\end{gather*}
Denote $\mathcal S = \cup_{r > 0}\PW_{[0,r]}$. We have $\mathcal S\subset L^2(\R, \sigma)$ if and only if $\sup_{x}\sigma([x, x + 1]) < \infty$ and in this case $L^2(\R,\sigma)$-closure of $\mathcal{S}$ coincides with the  $L^2(\R,\sigma)$-closure of $\mathcal{X}$, see \cite{lin1965equivalent} or Section \ref{sect: paley-wiener embedding} below.
In general the latter may not hold, however the divergence of the integral in \eqref{eq: szego assertion} implies the fact that $\mathcal S\cap L^2(\R, \sigma)$ is dense in $L^2(\R, \sigma)$. 
\medskip

The main object of the present paper is the Christoffel function
\begin{gather}
    \label{eq: christoffel function def}
    m_r(z) = m_r(\sigma, z)=\inf\left\{\|f\|_{L^2(\mathbb{R},\sigma)}\colon f\in \PW_{[0,r]},\ f(z)= 1\right\}, \qquad r\ge 0, \quad z\in\Cm.
\end{gather}
Lemma 8.4 in \cite{Denisov2006} states that for every $z\in \Cm_+ = \{w\colon \Im w > 0\}$ we have
\begin{gather*}
    \inf_{f\in \mathcal{S}}\left\|f - \frac{2\Im z}{x-z}\right\|_{L^2(\R, \sigma)} 
    =
    \dist\left(\frac{2\Im z}{x-z},\mathcal{S}\cap L^2(\R, \sigma)\right)_{L^2(\R, \sigma)}
    =\lim_{r\to \infty}m_r(z).
\end{gather*}
In particular, \eqref{eq: szego assertion} implies the convergence $m_r(z)\to 0$ as $r\to\infty$ for all $z\in \Cm_+$.
In the present paper we establish quantitative versions of this result.
An illustrative example is provided in Theorem \ref{thm: single deep zero} below. Roughly speaking, it states that if $d\sigma(x) = e^{-h(|x|)}\,dx$, where $h$ is a positive continuous  decreasing function on $(0, +\infty)$ that satisfies $\int_{\R_+} \frac{h(x)\, dx}{1 + x^2} = \infty$ then 
\begin{gather}
\label{eq: example ineq}
    \frac{1}{C}\int_\R\min(h(|x|), B)\frac{\Im z}{|x - z|^2}\,dx \le |\log m_r(z)|\le C\int_\R\min(h(|x|), B)\frac{\Im z}{|x - z|^2}\,dx, 
\end{gather}
where $C$ is some constant and $B > 0$ is chosen such that $rh^{-1}(B) = B$. 

The results of the present paper adapt the methods of the paper \cite{Borichev2020} by Borichev, Kononova, and Sodin, where the related question of weighted polynomial approximation on the unit circle was studied. Similar ideas were used in \cite{kononova2021convergence} to study a quantitative version of the Bernstein weighted approximation problem.
\smallskip

Suppose that $\PW_{[0,r]}\subset L^2(\R, \sigma)$. This embedding is automatically continuous and $\PW_{[0,r]}$ equipped with the $L^2(\R,\sigma)$-norm is a reproducing kernel Hilbert space if and only if the $L^2(\R)$-- and $L^2(\R, \sigma)$--norms are equivalent on $\PW_{[0,r]}$, see Section \ref{sect: paley-wiener embedding}.
In this case the reproducing kernel $K_r(\cdot, z)\in \PW_{[0,r]}$ at the point $z$ satisfies
\begin{gather*}
    f(z) = \langle f,K_r(\cdot,z)\rangle_{L^2(\R, \sigma)},\qquad f\in \PW_{[0,r]}.
\end{gather*}
The Cauchy--Schwarz inequality gives
\begin{gather*}
    |f(z)|^2
    \le
    \|f\|_{L^2(\R,\sigma)}^2
    \|K_r(\cdot,z)\|_{L^2(\R,\sigma)}^2 = \|f\|_{L^2(\R,\sigma)}^2
    K_r(z,z).
\end{gather*}
Moreover, the equality holds if and only if $f$ is proportional to $K_r(\cdot, z)$.  
Consequently, we have 
\begin{gather*}
    m_r(z) =\|K_r(\cdot,z)\|^{-1}_{L^2(\R,\sigma)} = \frac{1}{\sqrt{K_r(z, z)}}
\end{gather*}
and the infimum in \eqref{eq: christoffel function def} is attained.
Thus, the rate of decay of $m_r$ also describes the rate of growth of the norm of the reproducing kernel.
\medskip

Our interest in the asymptotic behaviour of Christoffel functions is motivated by the spectral theory of one-dimensional Dirac operators, where $\sigma$ arises as a spectral measure of the operator. The Paley-Wiener space $\PW_{[0,r]}$ equipped with the $L^2$ norm generated by such measures is always a reproducing kernel Hilbert space, see Theorem 1 in \cite{Bessonov2016} and Section 3 in \cite{Makarov2023}. Christoffel functions of spectral measures with convergent logarithmic integral were studied in \cite{Denisov2006} by Denisov. The author's papers \cite{Gubkin2024Dirac}, \cite{Gubkin2024Krein} established the connection between the rapid convergence of $m_r$ and the oscillation of the potential of Dirac operator, also see \cite{Gubkin2021}, where the behaviour of Christoffel functions on the real line was considered.
In \cite{Eichinger2024} Eichinger derived asymptotics of the Christoffel functions of the Schr\"{o}dinger operator.
\medskip

\paragraph*{\textbf{Notation.}}
Throughout the paper we use the notation $A\ls B$ when $A\le cB$ holds with some constant $c > 0$. The symbol $\simeq$ will be used if both $\ls$ and $\gs$ take place. We denote by $P_z$ the Poisson kernel at $z$ in the upper half-plane $\Cm_+$, 
\begin{gather*}
 P_z(t) =\frac{\Im z}{\pi|z - t|^2}.
\end{gather*}
Finally, given a real-valued function $f$ and a number $B$ we write $f_B = \min(f, B)$. For instance, with this notation, \eqref{eq: example ineq} rewrites as
\begin{gather*}
    |\log m_r(z)|\simeq\int_\R h_B(|x|) P_z(x)\,dx. 
\end{gather*}
Throughout the rest of the paper we assume that $\sigma$ is a measure with the divergent logarithmic integral \eqref{eq: szego assertion} and $r$ is sufficiently large so that $m_r(z) \le 1$ and $|\log m_r(z)| = -\log m_r(z)$.
\medskip

\paragraph{\textbf{Acknowledgements.}} I am grateful to Aleksei Kulikov for drawing my attention to the function used in the construction in Section \ref{section: comb domains}.

\section{Upper bounds}
Amrein--Berthier theorem, see Theorem II in the paper \cite{Nazarov1993} by Nazarov, states that there exists a constant $A > 0$ such that
\begin{gather}
\label{eq: Nazarov theorem}
    \|f\|_{L^2(\R)}^2\le Ae^{A\mu (E)\mu (\Sigma)}\left(\int_{\R\setminus E}|f|^2 + \int_{\R\setminus \Sigma}|\hat f|^2\right)
\end{gather}
holds for any sets $E, \Sigma\subset \R$ of finite measure and $f\in L^2(\R)$. Here and below $\mu$ denotes the Lebesgue measure on the real line.
In particular, when $f\in \PW_{[0,r]}$ this gives the Remez-type inequality
\begin{gather}
\label{eq: remez for PW}
    \|f\|_{L^2(\R)}^2\le Ae^{Ar\mu (E)}\int_{\R\setminus E}|f|^2\, dx.
\end{gather}
Assume that $\sigma$ is such that $\sigma' > 0$ a.\,e.\,on $\R$  and for some $a > 0$ the set $E_a = \{x\colon \sigma'(x) < a \}$ has finite Lebesgue measure.
Notice that in this case we have 
\begin{gather*}
    \lim_{a\to 0} \mu(E_a) = \mu\big(\{x\colon \sigma'(x) = 0\}\big) = 0. 
\end{gather*}
For small $a > 0$ and $f\in \PW_{[0,r]}$ inequality \eqref{eq: remez for PW} gives
\begin{gather*}
    \|f\|_{L^2(\R, \sigma)}^2 = \int_\R|f|^2d\sigma\ge a \int_{\R\setminus E_a}|f|^2\,dx 
    \ge \frac{ae^{-Ar\mu(E_a)}}{A}\|f\|_{L^2(\R)}^2.
\end{gather*}
If we additionally assume $f(z) = 1$ for $z\in \Cm_+$ then $\displaystyle\|f\|_{L^2(\R)}^2\gs r^{-1}$ and 
\begin{gather*}
    m_r^2(z)\gs \frac{a}{Ar}e^{-Ar\mu(E_a)}, 
    \\
    \liminf_{r\to\infty} m_r(z)^{2/r} \ge e^{-A\mu(E_a)}.
\end{gather*}
Letting $a\to 0$ we get $\liminf_{r\to\infty} m_r(z)^{1/r} \ge 1$.
The value of $m_r$ decreases in $r$ hence we actually get
\begin{gather*}
    \lim_{r\to\infty} m_r(z)^{1/r} = 1, \qquad z\in \Cm_+.
\end{gather*}
This is a continuous version of the classical Erd\H{o}s-Tur\'{a}n regularity criterion from the theory of orthogonal polynomials, see Corollary 4.1.3 in the book \cite{bookStahlTotik}. Equivalently this rewrites as 
 \begin{gather}
     \label{eq: erdos-turan regularity}
     |\log m_r(z)| = o(r),\quad r\to\infty.
 \end{gather}
On the other hand, if $\supp\sigma\neq \R$ then $m_r$ decays at least exponentially, see Corollary \ref{corol: measure with small support} below. The following theorem provides the quantitative version of \eqref{eq: erdos-turan regularity}. 
\begin{thm}
\label{thm: upper bound}
Let $\sigma$ be a measure with divergent logarithmic integral and
let $H$ be a measurable function on the real line such that \begin{gather}
    \label{eq: lower ineq for measure}
    d\sigma \ge \frac{e^{-H}dt}{1+t^2}.
\end{gather}
Suppose that the mapping  $a\mapsto\mu\big(\{x\colon H(x) > a\}\big)$ is finite and continuous for all sufficiently large $a$, then for all $z$ with $\Im z \ge 1$ we have 
\begin{gather*}
    \bigl|\log m_r(z)\bigr| \ls \log |z| +  \int_{\mathbb{R}} H_B(t)P_z(t)\,dt,
\end{gather*}
where $B$ is chosen such that $\mu\big(\{x\colon H(x) > B\}\big) = B/r$ and $H_B(t) = \min(H(t), B)$.
\end{thm}

\begin{proof}
Fix $z \in\Cm_+$ and consider $f\in \PW_{[0,r]}$ with $f(z) = 1$. 
We have
\begin{gather*}
    0 = \log |f(z)|^2 \le \int_\R\log |f|^2 P_z(t)\,dt.
\end{gather*}
Split the latter integral into three terms as follows:
\begin{align}
\label{eq: poisson in i}
\nonumber
0 \le \int_\R\log |f|^2 P_z(t)\,dt
   &= \int_{\{H>B\}} \log|f|^2 P_z(t)\,dt
   \\
   &+ \int_{\{H\le B\}} \log\bigl(|f|^2 e^{-H}\bigr)P_z(t)\,dt
   + \int_{\{H\le B\}} {H}P_z(t)\,dt.  
\end{align}
Denote
\begin{gather*}
    \lambda = \lambda_H(B, z) =\int_{\{H> B\}}P_z(t)\,dt, \qquad \lambda' = \lambda'_H(B, z) =\int_{\{H\le B\}}P_z(t)\,dt.
\end{gather*}
Both $\lambda$ and $\lambda'$ are positive and $\lambda' + \lambda = 1$.
For the second integral in \eqref{eq: poisson in i} we apply Jensen's inequality. It gives
\begin{align*}
    \int_{\{H\le B\}} \log\bigl(|f|^2 e^{-H}\bigr)P_z(t)\,dt &=\lambda' \int_{\{H\le B\}} \log\bigl(|f|^2 e^{-H}\bigr)\frac{P_z(t)\,dt}{\lambda'}
    \\
   &\le \lambda' \log\!\left(
      \int_{\{H\le B\}} |f|^2 e^{-H}\,\frac{P_z(t)\,dt}{\lambda'}
   \right)
   \\
&\le\lambda'\log\frac{1}{\lambda'} + \lambda'\log
      \int_{\R} |f|^2 e^{-H}P_z(t)\,dt.
\end{align*}
Harnack's inequality gives 
\begin{gather}
    \label{eq: Harnack for z and i}
    P_z(t) \le \frac{1 + \rho(z, i)}{1 - \rho(z, i)}P_i(t) =\frac{|z + i| + |z - i|}{|z + i| - |z - i|}P_i(t)\le \frac{(|z| + 1)^2}{\Im z}P_i(t).
\end{gather}
Here $\rho(w, z) = \left|\frac{z - w}{z - \ol{w}}\right|$ is the pseudohyperbolic distance in $\Cm_+$. Denote $C(z) = \frac{(|z| + 1)^2}{\Im z}$. We also have $\lambda'\log \frac{1}{\lambda'}\le 1/e$  hence
\begin{align*}
    \int_{\{H\le B\}} \log\bigl(|f|^2 e^{-H}\bigr)P_z(t)\,dt 
    &\le
    \frac{1}{e} +\lambda'\log C(z) + \lambda'  \log
      \int_{\R} |f|^2 e^{-H}P_i(t)\,dt
      \\
      &\le \frac{1}{e} +\lambda'\log C(z) + \lambda'  \log
      \int_{\R} |f|^2\,d\sigma.
\end{align*}
Similarly, for the first integral in \eqref{eq: poisson in i} we get
\begin{align*}
\int_{\{H>B\}} \log|f^2|P_z(t)\,dt
\le
    \frac{1}{e} + \lambda\log\left(\int_{\mathbb{R}} |f|^{2}P_z(t)\,dt\right).
\end{align*}
Let $g=\dfrac{f}{x-z}$. The Fourier transform of $g$ is given by
\begin{gather*}
\widehat{g}(\xi)
=
\begin{cases}
\displaystyle i e^{-iz\xi}\int_0^r e^{izt}\widehat{f}(t)\,dt,
& \xi<0,\\[3mm]
\displaystyle i e^{-iz\xi}\int_\xi^r e^{izt}\widehat{f}(t)\,dt,
& 0\le \xi\le r,\\[3mm]
0,
& \xi>r.
\end{cases}
\end{gather*}
It decays exponentially as $\xi\to-\infty$ hence
\begin{gather}
\label{eq: l2 norm decay}
\|\widehat{g}\|_{L^2(\mathbb{R}\setminus[-r,r])}
=
e^{-r\Im z}\|\widehat{g}\|_{L^2(\mathbb{R}\setminus[0,r])}
\le
e^{-r\Im z}\|g\|_{L^2(\mathbb{R})}.
\end{gather}
Now let $E_B=\{x:H(x)\geq B\}$,\, $\Sigma=[-r,r]$ and recall that $B$ satisfies 
\begin{gather*}
    r\mu(E_B) = B.
\end{gather*}
Amrein-Berthier theorem \eqref{eq: Nazarov theorem} for $g$ gives
\begin{align}
\nonumber
\int_{\mathbb{R}} |f|^{2}P_z(t)\,dt=\frac{\Im z}{\pi}\int_{\mathbb{R}}|g|^2\,dt
&\le
\frac{\Im z}{\pi}A e^{2A\mu(E_B)r}
\left(
\int_{\{H<B\}}|g|^2\,dt
+
\int_{\R\setminus\Sigma}|\widehat{g}|^2\,d\xi
\right)
\\
\label{eq: nazarov for g}
 &= Ae^{2AB}\frac{\Im z}{\pi}\left(\int_{\{H<B\}}|g|^2\,dt + \int_{\R\setminus\Sigma}|\widehat{g}|^2\,d\xi\right).
\end{align}
For the first term we write
\begin{gather*}
\frac{\Im z}{\pi}\int_{\{H<B\}}|g|^2\,dt
=
\int_{\{H<B\}}|f|^2P_z(t)\,dt
\le 
e^B\int_{\R}|f|^2P_z(t)e^{-H}\,dt.
\end{gather*}
Harnack's inequality \eqref{eq: Harnack for z and i} and assertion \eqref{eq: lower ineq for measure} imply
\begin{align*}
\frac{\Im z}{\pi}\int_{\{H<B\}}|g|^2\,dt
\le 
e^BC(z) \int_{\R}|f|^2P_i(t)e^{-H}\,dt 
\le 
e^BC(z)\int_{\mathbb{R}}|f|^2\,d\sigma.
\end{align*}
This, \eqref{eq: l2 norm decay} and \eqref{eq: nazarov for g} give
\begin{gather*}
    \int_{\mathbb{R}} |f|^{2}P_z(t)\,dt
    \le 
    Ae^{2AB}\left(e^BC(z) \int_{\mathbb{R}}|f|^2\,d\sigma + e^{-2r\Im z}\int_{\mathbb{R}} |f|^{2}P_z(t)\,dt\right).
\end{gather*}
Since $B = o(r)$ as $r\to\infty$, the second term can be absorbed into the left-hand side and
\begin{gather*}
    \int_{\mathbb{R}} |f|^{2}P_z(t)\,dt\le 2Ae^{2AB}e^BC(z) \int_{\mathbb{R}}|f|^2\,d\sigma.
\end{gather*}
Substituting the obtained estimates into \eqref{eq: poisson in i} we get
\begin{align*}
    0&\le \left(\frac{1}{e} + \lambda \log\Big(2Ae^{2AB}e^BC(z) \int_{\mathbb{R}}|f|^2\,d\sigma\Big)\right)
    \\
    &+
    \left(\frac{1}{e} +\lambda'\log C(z) + \lambda'  \log
      \int_{\R} |f|^2\,d\sigma\right)
      +
      \int_{\{H\le B\}} {H}P_z(t)\,dt.
\end{align*}
We notice that $1 + \log|z|\gs\log C(z)$ hence
\begin{gather*}
    -\log
      \int_{\R} |f|^2\,d\sigma\ls 1 + \log |z| + B\lambda_H(B,z) + \int_{\{H\le B\}} {H}P_z(t)\,dt = \log |z| + \int_{\R} {H_B}P_z(t)\,dt.
\end{gather*}
Recall definition \eqref{eq: christoffel function def} of the Christoffel function.
The latter inequality holds for all $f\in \PW_{[0,r]}$ with $f(z) = 1$ therefore
\begin{gather*}
   -\log m_r(z) = \sup\left\{-\log\|f\|_{L^2(\mathbb{R},\sigma)}:f\in \PW_{[0,r]},\ f(z)= 1\right\}  \ls \log |z| + \int_{\R} {H_B}P_z(t)\,dt.
\end{gather*}
The proof is concluded.
\end{proof}

\section{Lower bounds}

\subsection{Comb domains and Martin functions}\label{section: comb domains}
Let $E\subset \R$ be a closed set such that $\R\setminus E$ is a union of $k > 0$ disjoint open bounded intervals. There exist real $u_1 < u_2 < \ldots < u_k$ and positive numbers $h_1, \ldots, h_k$ such that  $E = \Theta^{-1}(\R)$, where $\Theta$ is the conformal mapping from $\Cm_+$ to the domain
\begin{gather*}
\Omega
=
\mathbb C_+ \setminus \bigcup_{m=1}^k \left\{ \pi u_m+iy:\ 0<y\le h_m \right\}.
\end{gather*}
The domains of this type are called comb domains, for more details see \cite{Eremenko2012} and Section 7 in \cite{sodin1992functions}. 
The function $\Theta$ has positive imaginary part hence it admits the Herglotz representation 
\begin{gather}
\label{eq: Herglotz representation}
\Theta_\Omega(z)
=
az+b+
\int_{\mathbb R}
\left(
\frac{1}{t-z}
-
\frac{t}{1+t^2}
\right)
\,d\rho(t),
\qquad
a\ge0,\quad b\in\mathbb R.
\end{gather}
The Stieltjes inversion formula (see Theorem 2.2 in \cite{Gesztesy2000}) shows that the measure $\rho$ has compact support and consequently we have
\begin{gather*}
    \Theta(z) = az + b + O(1/|z|), \qquad |z|\to\infty.
\end{gather*}
In particular, the parameter $a$ in our situation is strictly positive. 
The function $M(z) = \Im \Theta(z)$ is called the symmetric Martin function of $E$. The Schwarz reflection principle shows that its symmetric extension $M(\ol{z}) = M(z)$ is harmonic in $\Cm\setminus E$.
\medskip

We will call $E$ comb-regular if all $u_m$ can be chosen to be integers. In this case for every integer $n \ge 0$ the function 
\begin{gather}
\label{eq: function A(z)}
    A(w) = \cos(n\Theta(w))
\end{gather}
attains real values on $\R$ hence it extends analytically to the whole complex plane. From \eqref{eq: Herglotz representation} it follows that $A$ is an entire function of exponential type $na$.
For every $x\in \R$ we have
\begin{gather*}
    |A(x)|\le e^{n\Im\Theta(x)} \le e^{nh},
\end{gather*}
where 
\begin{gather}
    \label{eq: h of domain}
     h = h_{\Omega} = \max_{1\le m\le k} h_m.
\end{gather}
Moreover, the inequality $|A(x)|\le 1$ holds when $x\in E$.
\begin{thm}\label{thm: Martin function}
Assume that the non-negative function $H$ satisfies 
\begin{gather*}
    \int_\R\frac{e^Hd\sigma(x)}{x^2 + 1} < \infty.
\end{gather*}
    Let $E$ be a comb-regular set such that $\R\setminus E\subset \{x\colon H(x) >  B\}$ and let $M$ be the corresponding Martin function. Define the numbers $a$ and $h$ by \eqref{eq: Herglotz representation} and \eqref{eq: h of domain}. If $4(a + 1)\le r$ and $2rh/a\le B$ then for $z$ with $\Im z\ge 1$ we have 
    \begin{gather*}
        |\log m_r(z)|\gs r\left(\frac{M(z)}{a} - \Im z\right) - \log |z|.
    \end{gather*}
\end{thm}
\begin{proof}
    Fix $\alpha = 1$, let $\rho = r/2 - \alpha$ and let $n = [\rho /a]$ be the largest integer not greater than $\rho/a$. We have 
    $n = [\rho/a]\ge \frac{\rho - a}{a} = \frac{r/2 - a - 1}{a}\ge \frac{r}{4a}$. Define $A(w) = \cos(n \Theta(w))$ and put
    \begin{gather*}
        F(w) = A(w) \frac{\sin(\alpha w)}{\alpha w} e^{i(na + \alpha) w
        }.
    \end{gather*}
The Paley-Wiener theorem implies $F\in \PW_{[0,2r']}$ with $r' = na + \alpha\le r/2$ hence
\begin{gather}
    \label{eq: Martin mr estimate}
    m_r^2(z)\le \|F\|^2_{L^2(\R,\sigma)}/ |F(z)|^2.
\end{gather}
We have 
\begin{gather*}
    \int_{\mathbb R}|F|^2\,d\sigma
 =
 \int_{\mathbb R}
 |F|^2e^{-H}(1 + x^2)\frac{e^{H}\,d\sigma(x)}{1 + x^2}
 \ls \sup_{x\in \R} \left(|F|^2e^{-H}(1 + x^2)\right).
\end{gather*}
For every $x\in \R$ we have $\left|\frac{\sin(\alpha x)}{\alpha x}e^{ir'x}\right|\ls \frac{1}{|x| + 1}$ hence $(x^2 + 1)|F(x)|^2\le |A(x)|^2$ holds. This means 
\begin{gather*}
    \begin{cases}
        (x^2 + 1)|F(x)|^2\ls 1, & x\in E,
        \\
        (x^2 + 1)|F(x)|^2\ls e^{2n h}, &x\in \R\setminus E.
    \end{cases}
\end{gather*}
Recall that $H\ge B$ on $\R\setminus E$ and $H\ge 0$ on the whole real line. We have $2nh\le 2rh/a\le B$ hence 
\begin{gather*}
    \int_{\mathbb R}|F|^2\,d\sigma\ls \max (1, e^{2n h - B})\le 1.
\end{gather*}
Further we write
\begin{gather*}
    |F(z)| = \frac{|\sin(\alpha z)|}{|\alpha z|}|A(z)|e^{-r'\Im z} 
    \gs
    \frac{1}{|z|}\left(e^{\alpha\Im z} - e^{-\alpha\Im z}\right)\left(e^{n M(z)} - e^{-n M(z)}\right)e^{-r'\Im z}.
\end{gather*}
We have $nM(z)\ge na\Im z \ge r/4\gs 1$ hence
\begin{gather*}
    |F(z)|\gs \frac{1}{|z|}e^{n (M(z) - a\Im z)}. 
\end{gather*}
Substitute the obtained bounds into \eqref{eq: Martin mr estimate}. We get
\begin{gather*}
    m_r^2(z) \ls |z|^2 e^{-2n (M(z) - a\Im z)}.
\end{gather*}
To conclude the proof we take the logarithm and use $n\ge \frac{r}{4a}$:
\begin{gather*}
    |\log m_r^2(z)| \gs n (M(z) - a\Im z) - \log |z|\gs r\left(\frac{M(z)}{a} - \Im z\right) - \log |z|.
\end{gather*}

\end{proof}

\begin{corol}\label{cor: Martin single interval}
Let $L =  (c-\ell, c+ \ell)$ be an interval such that $\ell \le 1$ and $H(x) > B$ for every $x\in L$. If $\Im z\ge 1$,  $r\ge 8$ and $r\le B/(2\ell)$ then 
\begin{gather*}
    |\log m_r(z)|\gs \frac{r\ell^2\Im z}{|z - c|^2} -\log|z|.
\end{gather*}
\end{corol}
\begin{proof}
The function $\Theta(z) = \sqrt{(z - c)^2 - \ell^2}$ is the conformal mapping from $\Cm_+$ to the comb domain $\Omega = \Cm_+\setminus\{iy,\,0 <y\le \ell \}$ and $\Theta^{-1}(\R) = \R\setminus L$. It follows that $E = \R\setminus L$ is a comb-regular set. The parameters $a$ and $h$ from \eqref{eq: Herglotz representation} and \eqref{eq: h of domain} satisfy $a = 1$ and $h = \ell$. By Theorem \ref{thm: Martin function},
\begin{gather*}
    |\log m_r(z)|\gs r(M(z) - \Im z) - \log|z|, \qquad 2r\le B/\ell,
\end{gather*}
where 
\begin{gather*}
    M(z) = \Im \Theta(z) = \Im \sqrt{(z - c)^2 - \ell^2}.
\end{gather*}
Denote $w = z -c = x + iy$ then we have
\begin{gather*}
    (z - c)^2 - \ell^2 = w^2 - \ell^2 = x^2 - y^2 - \ell^2 + 2ixy.
\end{gather*}
Explicit formula
\begin{gather*}
    \left(\Im\sqrt{P + iQ}\right)^2 = \frac{-P + \sqrt{P^2 + Q^2}}{2} = \frac{-P + |P + iQ|}{2}
\end{gather*}
shows
\begin{gather*}
    M(z)^2 = \left(\Im\sqrt{(z - c)^2 - \ell^2}\right)^2 = \frac{- \big(x^2 - y^2 - \ell^2\big) + |w^2 - \ell^2|}{2},
    \\
    M(z)^2 - y^2 = \frac{|w^2 - \ell^2| - |w|^2 + \ell^2}{2} = \frac{2\ell^2 y^2}{|w^2 - \ell^2| + |w|^2 - \ell^2}\ge \frac{\ell^2 y^2}{|w|^2}.
\end{gather*}
Transforming we obtain
\begin{gather*}
    M(z) - y\ge y\left(\sqrt{1 + \ell^2/|w|^2} - 1\right).
\end{gather*}
Since $\ell\le 1$ and $|w|^2\ge (\Im z)^2\gs 1$ we have $\ell^2/|w|^2\le 1$ hence
\begin{gather*}
    M(z) - y\gs \frac{y\ell^2}{|w|^2} = \frac{y\ell^2}{|z - c|^2}.
\end{gather*}
This completes the proof.
\end{proof}    
\begin{corol}\label{corol: measure with small support}
If the support of $\sigma$ does not fill the whole real line then  $|\log m_r(z)|\gs r$ holds for all large $r$. 
\end{corol}

\subsection{Approximation by outer functions}
Let $\phi$ be an unbounded continuous decreasing function on $(0, +\infty)$.
We call the function $H$ subordinated to $\phi$ if $\displaystyle\lim_{x\to\infty}{\phi(x)}\le \inf_{x\in\R}H(x)$ and 
\begin{gather*}
\phi^{-1}\circ H
\end{gather*}
is a non-negative Lipschitz function with Lipschitz constant at most one.
\smallskip

We call the unbounded continuous decreasing function $\varphi$ on
$(0,+\infty)$ \emph{regular} if it satisfies at least one of the following
two conditions:
\begin{gather}
\label{Reg1}\tag{Reg1}
\text{the function }x\mapsto x\varphi(x)\text{ does not decrease
and }
\varphi(x)\gs\log\frac1x,
\\
\label{Reg2}\tag{Reg2}
\varphi(x/2)\ls\varphi(x)
\quad\text{and}\quad
\varphi(x)\gtrsim\frac1x.
\end{gather}
These regularity assertions are taken from \cite{Borichev2020}, see Section 4 there. 
\smallskip

Consider the following class of  measures
\begin{gather*}
    \Meas = \{\sigma\colon \sigma'\in L^1(\R) + L^\infty(\R),\,\, \sigma_s(\R) < \infty\},
\end{gather*}
where $\sigma'$ and $\sigma_s$ are the density of the a.\,c.\,part of $\sigma$ and the singular part of $\sigma$ with respect to the Lebesgue measure on the real line. 
If $\sigma\in\Meas$ then there exist $\sigma_1\in L^1(\R)$, $\sigma_\infty\in L^\infty(\R)$ such that $\sigma' = \sigma_1 + \sigma_\infty$ and for $f\in L^2(\R)\cap L^\infty(\R)$ we have 
\begin{align}
    \nonumber
    \|f\|_{L^2(\R, \sigma)}^2 = \int_{\R}|f|^2\,d\sigma 
    &= \int_{\R}|f|^2\sigma_1\,dx + \int_{\R}|f|^2\sigma_\infty\,dx + \int_{\R}|f|^2\,d\sigma_s
    \\
    \label{eq: L2 sigma in terms of L2 and Linfty}
    &\le 
    \|f\|_{L^\infty(\R)}^2 \left(\|\sigma_1\|_{L^1(\R)} + \sigma_s(\R)\right) + \|f\|_{L^2(\R)}^2\|\sigma_\infty\|_{L^\infty(\R)}. 
\end{align}

\begin{thm}\label{thm: theorem lower bound}
Assume that $\sigma\in\Meas$ and
\begin{gather*}
\int_\R\frac{e^Hd\sigma(x)}{x^2 + 1} < \infty
\end{gather*}
with $H$ subordinated to a regular function $\varphi$. Then for all $z$ with $\Im z\ge 1$ and $r\gs \log |z|$ we have
\begin{gather*}
|\log m_r(z)| \gs \int_{\R}H_BP_z\,dx - \log |z|,
\end{gather*}
where $B$ solves the equation
$r\varphi^{-1}(B)=B$
when $\varphi$ satisfies condition \eqref{Reg1}, and $B=\sqrt{r}$
when $\varphi$ satisfies condition \eqref{Reg2}.
\end{thm}
If $g$ is a positive measurable function on $\R$ such that $\int_\R\frac{|\log g(t)|\,dt}{1 + t^2} < \infty$ then, see Section 4 in \cite{Garnett}, there exists an outer function $G$ in $\Cm_+$ with nontangential boundary values that satisfy $|G(x)| = g(x)$ a.\,e.\,on $\R$. It is given by the explicit formula
\begin{gather}
    \label{eq: outer function explicit}
    G(w) = \exp\left(\frac{1}{\pi i}\int_\R \log g(t)\,\left(\frac{1}{t - w}- \frac{t}{1 + t^2}\right)\,dt \right), \qquad w\in \Cm_+.
\end{gather}
In this case we also have 
\begin{gather}
\label{eq: outer function modulus}
    \log |G(w)| = \int_\R \log g(t)\, P_w(t)\,dt.
\end{gather}
\begin{lemma}\label{lemma: outer function}
Let $g\ge 1$ be a function on $\R$ such that $\int_\R\frac{|\log g(t)|\,dt}{1 + t^2} < \infty$ and let $G$ be the corresponding outer function. Then for $z \in \Cm_+$ and $h\ge 0$ we have 
\begin{gather*}
    |G(z + ih)|\ge |G(z)|^{\frac{\Im z}{\Im z + h}}.
\end{gather*}
\end{lemma}
\begin{proof}
The function
\begin{gather*}
    U(w) = \log|G(w)| = \int_\R \log g(t) P_w(t)\,dt
\end{gather*}
is non-negative and harmonic in $\Cm_+$. Harnack's inequality, recall \eqref{eq: Harnack for z and i}, for $U$ gives
\begin{gather*}
    U(z + ih)\ge \frac{1 - \rho(z, z + ih)}{1 + \rho(z, z + ih)}U(z) = \frac{\Im z}{\Im z + h} U(z),
\end{gather*}
where $\rho(z,w) = \left|\frac{z - w}{z - \ol w}\right|$ is the pseudohyperbolic distance in $\Cm_+$.
It follows that
\begin{gather*}
    |G(z + ih)| =e^{U(z + ih)} \ge  e^{\frac{\Im z}{\Im z + h} U(z)} = |G(z)|^{\frac{\Im z}{\Im z + h}} .
\end{gather*}

\end{proof}

\begin{lemma}\label{lemma: main lemma lower bound}
Assume that $\sigma \in \Meas$ and let $H\ge 0$ be such that
\begin{gather*}
    \int_\R\frac{e^Hd\sigma(x)}{x^2 + 1} < \infty, \qquad \qquad H_B * P_{ih} \ls  H
\end{gather*}
with $B$ and $h\le 1$ that satisfy $\log h^{-1}\ls B$. 
Then for all $z$ with $\Im z \ge 1$ and $r\gs B/h$, $r\gs \log|z|$ we have 
\begin{gather*}
    |\log m_r(z)| \gs  \int_{\R}H_B(t)P_z(t)\,dt - \log |z|.
\end{gather*}
\end{lemma}
\begin{proof}
Assume that the inequality
\begin{gather*}
    \int_\R H_B(t)P_{x + ih}(t)\,dt  = (H_B * P_{ih})(x) \le MH(x)
\end{gather*}
holds with some constant $M\ge 1$ for all $x\in \R$.
Let $F$ be an outer function with the boundary values $|F| = \exp(H_B / (2M))$ a.\,e.\,on $\R$.  Formula \eqref{eq: outer function modulus} shows 
\begin{gather*}
    1\le |F(w)|\le \exp(B/(2M)), \qquad w\in \Cm_+.
\end{gather*}
Lemma \ref{lemma: outer function} for $F$ gives
\begin{gather*}
    |F(z + ih)|\ge |F(z)|^{\Im z/(\Im z + h)}.
\end{gather*}
The inequality $|F(z)|\ge 1$ then implies $|F(z + ih)|\gs \sqrt{|F(z)|}$. 
Since $F$ is bounded in $\Cm_+$ we have $T(w) = \frac{F(w)}{w + i}\in H^2(\Cm_+)$ and $T_h(w) = T(w + ih)$ satisfies
\begin{gather*}
    |T_h(z)| = |T(z + ih)| = \frac{|F(z + ih)|}{|z + i + ih|}\gs \frac{\sqrt{|F(z)|}}{|z|} = \frac{1}{|z|}\exp\left(\frac{1}{4M} \int_\R H_B P_z\,dt\right).
\end{gather*}
We can also estimate
\begin{gather*}
    \log |F(x + ih)| = \frac{1}{2M}\int_\R H_B(t)P_{x + ih}(t)\,dt\le \frac{H(x)}{2},
    \\
    \int_\R |T_h(x)|^2d\sigma(x) = \int_\R\frac{|F(x + ih)|^2\, d\sigma(x)}{x^2 + (1 + h)^2} 
    \le
    \int_\R\frac{\exp\left(H(x)\right)}{x^2 + 1}\,d\sigma(x) \ls 1.
\end{gather*}
From the obtained bounds we get
\begin{gather*}
    \log\left(\|T_h / T_h(z)\|_{L^2(\R,\sigma)}\right)\ls -\frac{1}{4M} \int_\R H_B P_z\,dt + \log |z|,
\end{gather*}
however $T_h$ does not belong to $\PW_{[0,r]}$ and the desired inequality for $m_r$ does not follow immediately. Let us approximate $T_h$. We have $T\in H^2(\Cm_+)$ hence there exists $\phi\in L^2(\R_+)$ such that
\begin{gather*}
    T(w) = \frac{1}{\sqrt{2\pi}}\int_{\R_+} \phi(\xi)e^{i\xi w}\, d\xi,\qquad w\in \Cm_+,
    \\
    T_h(w) = T(w + ih) = \frac{1}{\sqrt{2\pi}}\int_{\R_+} \phi(\xi)e^{i\xi w - \xi h}\, d\xi,\qquad w\in \Cm_+.
\end{gather*}
From the Plancherel theorem we get
\begin{gather}
    \label{eq: norm of phi}
    \|\phi\|_{L^2(\R_+)} = \|T\|_{L^2(\R)} = \Big\|\frac{F}{x + i}\Big\|_{L^2(\R)}\le \Big\|\frac{e^{B/(2M)}}{x + i}\Big\|_{L^2(\R)}\ls e^{B/(2M)}.
\end{gather}
Introduce
\begin{gather*}
    G_{r}(w) = \frac{1}{\sqrt{2\pi}}\int_0^r \phi(\xi)e^{i\xi w - \xi h}\, d\xi, \qquad w\in \Cm.
\end{gather*}
By the definition we have $G_r\in \PW_{[0,r]}$ and
\begin{gather*}
    G_{r}(w) - T_h(w) = \frac{1}{\sqrt{2\pi}}\int_r^\infty \phi(\xi)e^{i\xi w - \xi h}\, d\xi, \qquad w\in \Cm_+.
\end{gather*}
The Plancherel theorem, \eqref{eq: norm of phi} and the assertion $2rh\ge B$ give
\begin{gather*}
    \|G_r - T_h\|^2_{L^2(\R)} = \int_r^\infty|\phi(\xi)|^2e^{-2\xi h}\, d\xi 
    \le
    e^{-2rh}\|\phi\|_{L^2(\R)}^2  
    \ls e^{-2rh +B/M} \ls 1.
\end{gather*}
We can also obtain the pointwise estimate on $G_r - T_h$ in $\R\cup\Cm_+$. 
Denote $I(w) = \Im w + h$.
It follows that
\begin{gather}
    \label{eq: G_r -T_h at point}
    |G_{r}(w) - T_h(w)| \ls \int_r^\infty |\phi(\xi)|e^{-I(w)\xi}\, d\xi
\le \|\phi\|_{L^2(\R_+)}\frac{e^{-rI(w)/2}}{\sqrt{I(w)}} 
    \ls \frac{e^{B/(2M)-rI(w)/2}}{\sqrt{I(w)}}.
\end{gather}
For $w = z$ we have $I(w)\ge \Im z\ge 1$ hence $\left|G_{r}(z) - T_h(z)\right| ^2
    \ls \exp\left(-r\Im z + B/M\right).$
We also have $|T_h(z)|\gs |z|^{-1}$ therefore for $r\gs \log |z|$ we get
\begin{gather*}
    |G_r(z)|\gs |T_h(z)| = \frac{1}{|z|}\exp\left(\frac{1}{4M} \int_\R H_B P_z\,dt\right).
\end{gather*}
Inequality \eqref{eq: G_r -T_h at point} on the real line and the assertions $\log h^{-1}\le B$, $rh\gs B$ give 
\begin{gather*}
    |G_r(x) - T_h(x)|\ls \frac{e^{B/(2M)-rh/2}}{\sqrt{h}} \ls 1, \qquad x\in \R.
\end{gather*}
Recall that we have $\sigma\in\Meas$ and \eqref{eq: L2 sigma in terms of L2 and Linfty} applies. It implies
\begin{gather*}
    \|G_r - T_h\|_{L^2(\R, \sigma)}\ls \|G_r - T_h\|_{L^2(\R)} + \|G_r - T_h\|_{L^\infty(\R)}\ls 1,
    \\
    \|G_r\|_{L^2(\R, \sigma)}\le \|G_r - T_h\|_{L^2(\R, \sigma)} + \|T_h\|_{L^2(\R, \sigma)}\ls 1.
\end{gather*}
Finally, we can write 
\begin{gather*}
    m_r(z) \le \frac{\|G_r\|_{L^2(\R, \sigma)}}{|G_{r}(z)|} \ls |z|\exp\left(-\frac{1}{4M} \int_\R H_B P_z\,dt\right).
\end{gather*}
Taking the logarithm we finish the proof.
\end{proof}

\subsection{Proof of Theorem \ref{thm: theorem lower bound}} 
The following two technical lemmas are essentially Lemmas 7 and 8 from \cite{Borichev2020}, we prove them in Section \ref{sec: proof of lemmas} below.

\newcommand{\firstlemma}{
Let $\phi:(0,\infty)\to(0,\infty)$ be an unbounded decreasing function, let
$\widetilde \phi$ be its even extension to $\mathbb R$ and $\widetilde\phi_B = \min(\widetilde \phi, B)$.
The inequality
\begin{gather*}
\widetilde \phi_B*P_{ih}\ls \widetilde \phi
\end{gather*}
holds everywhere on $\mathbb R$, provided that at least one of the following
conditions holds\textup{:}
\begin{enumerate}
\item[$(i)$] the function $x\mapsto x^2\phi(x)$ is nondecreasing and $h\ls \phi^{-1}(B)$\textup{;}
\item[$(ii)$] $\phi$ satisfies \eqref{Reg2} and $h\ls 1/B$.
\end{enumerate}
}
\begin{lemma}[Lemma 7, \cite{Borichev2020}]\label{lem: phi poisson}
\firstlemma
\end{lemma}
Notice that \eqref{Reg1} implies assertion $(i)$ of this lemma.

\newcommand{\secondlemma}{
Let $H$ be subordinated to a regular function $\phi$ and let $H_B=\min(H,B)$.
Then everywhere on $\mathbb R$ we have
\begin{gather*}
H_B*P_{ih}\ls H 
\end{gather*}
provided that
$h\ls\varphi^{-1}(B)$ when $\varphi$ satisfies condition \eqref{Reg1}, and
$h\ls B^{-1}$ when $\varphi$ satisfies condition \eqref{Reg2}.
}

\begin{lemma}[Lemma 8, \cite{Borichev2020}]
\label{lemma: poisson kernel of subordinated function}
\secondlemma
\end{lemma}

\begin{proof}[Proof of Theorem \ref{thm: theorem lower bound}]
Recall that we need to prove inequality
\begin{gather}
\label{eq: main ineq in thm 3}
|\log m_r(z)| \gs \int_{\R}H_BP_z\,dx - \log |z|,
\end{gather}
where $r\varphi^{-1}(B)=B$ when $\varphi$ satisfies condition \eqref{Reg1}, and $B=\sqrt{r}$
when $\varphi$ satisfies condition \eqref{Reg2}.
By Lemma \ref{lemma: poisson kernel of subordinated function} we have $H_B*P_{ih}\ls H$ if $h = \phi^{-1}(B)$ in the case of \eqref{Reg1} and $h = B^{-1}$ when \eqref{Reg2} holds. We claim that in both situations we have
\begin{gather*}
\log h^{-1} \ls B.
\end{gather*}
This is trivial when $h = B^{-1}$, and it reduces to
\begin{gather*}
    \log \frac{1}{\phi^{-1}(B)} \ls B
\end{gather*}
when $\phi$ satisfies \eqref{Reg1}.
The latter is equivalent to $\phi(x)\gs \log\frac{1}{x}$ given by \eqref{Reg1}.
Thus, for our parameters, Lemma \ref{lemma: main lemma lower bound} applies and gives \eqref{eq: main ineq in thm 3} with the required $B$. 
The proof is concluded.
\end{proof}
\section{Examples. Measures with deep zero at one point}

\begin{thm}\label{thm: single deep zero}
Let $h$ be a positive continuous decreasing function on $(0, \infty)$ such that
\begin{gather*}
\int_{\R_+} \frac{h(x)\,dx}{1 + x^2}=+\infty.
\end{gather*}
Suppose that $h$ satisfies at least one of the following two conditions\textup{:}
\begin{enumerate}
\item[$(i)$] the function $\theta\mapsto\theta^2h(\theta)$ does not decrease, and $|\log\theta|=O(h(\theta))$ as $\theta \to 0$\textup{;}
\item[$(ii)$] $\displaystyle\limsup_{a\to\infty}\frac{h^{-1}(a)}{h^{-1}(2a)} < 2$.
\end{enumerate}
Fix the point $x_0\in \R$, define $H = h(|x - x_0|)$ and let $\sigma$ be an absolutely continuous measure on $\R$ with the density $e^{-H(x)}$. Then for every fixed $z$ with $\Im z\ge 1$ and large $r$ we have
\begin{gather*}
|\log m_r(z)|
\simeq
\int_\R H_B(x)P_z(x)\, dx,
\end{gather*}
where $B$ solves the equation $r\,h^{-1}(B)=B$ and $H_B=\min(H,B)$ as before.
\end{thm}
\begin{proof}
The claim of the theorem follows from the following two estimates:
\begin{gather*}
|\log m_r(z)|
\ls
\int_{\mathbb R} H_B(x)P_z(x)\,dx,
\qquad
|\log m_r(z)|
\gs
\int_{\mathbb R} H_B(x)P_z(x)\,dx.
\end{gather*}
The first estimate is given by Theorem \ref{thm: upper bound} and does not require any
regularity assumptions on $h$. Conditions $(i)$ and $(ii)$ are only needed for the second estimate. Both conditions imply $\sigma'\in L^\infty(\R)$ and consequently $\sigma\in\Meas$. In the first case Lemma \ref{lem: phi poisson} applies and Lemma \ref{lemma: main lemma lower bound} gives the required  inequality.  
When condition $(ii)$ holds, we use Corollary \ref{cor: Martin single interval}: every $x\in (x_0 - h^{-1}(2B), x_0 + h^{-1}(2B))$ satisfies $H(x) > 2B$ and $r = \frac{B}{h^{-1}(B)}\le \frac{2B}{2h^{-1}(2B)}$ hence
\begin{gather*}
    |\log m_r(z)|\gs \frac{r(h^{-1}(2B))^2\Im z}{|z - x_0|^2}\simeq \frac{r(h^{-1}(B))^2\Im z}{|z - x_0|^2}\simeq B h^{-1}(B) P_z(x_0).
\end{gather*}
Fix a sufficiently small $\delta > 0$. We have 
\begin{gather*}
    \int_{\mathbb R} H_B(x)P_z(x)\,dx  = \int_{\mathbb R} H_B(x + x_0)P_z(x + x_0)\,dx \simeq \int_0^\delta H_B(x + x_0)P_z(x + x_0)\,dx + O(1)
    \\
    \simeq P_z(x_0)\left(Bh^{-1}(B) + \int_{h^{-1}(B)}^\delta H_B(x + x_0)\,dx\right) + O(1), \qquad B\to\infty.
\end{gather*}
To finish the proof we notice that the second term is bounded by a constant multiple of the first one. Indeed, splitting the integral into dyadic level sets of $h$, condition $(ii)$ implies that the resulting contributions form a geometrically decreasing series.
\end{proof}

\begin{corol}
Let $x_0\in\mathbb R$ and $z\in\mathbb C_+$ with $\Im z\ge 1$ be fixed and let $\sigma$ be an absolutely continuous measure on $\R$ with density $\exp\bigl(-h(|x-x_0|)\bigr)$.
Then, for all sufficiently large $r$, we have
\begin{enumerate}
    \item if $\displaystyle h(x)=x^{-1}\log^{-1}(1/x)$ for all sufficiently small $x > 0$  with any positive decreasing continuation to $(0, \infty)$, then $\displaystyle |\log m_r(z)|\simeq \log\log r$;
    \item if $ p>1$  and $\displaystyle h(x)=x^{-1}\log^p(1/x)$ for all sufficiently small $x > 0$  with any positive decreasing continuation to $(0, \infty)$, then $\displaystyle |\log m_r(z)|\simeq (\log r)^{p + 1}$;
    \item if $p > 1$ and $\displaystyle h(x)=x^{-p}$, then $\displaystyle |\log m_r(z)|\simeq r^{(p-1)/(p+1)}$;
    \item if $p > 0$ and $\displaystyle h(x)=\exp\left(x^{-p}\right)$, then $\displaystyle |\log m_r(z)|\simeq r(\log r)^{-2/p}$.
\end{enumerate}

\begin{proof}
Let $B$ be the solution of $r h^{-1}(B)=B$ and $h_{B}(x):=\min\bigl\{h(|x-x_0|),B\bigr\}$. Fix a sufficiently small $\delta>0$. Similarly to the proof of the previous theorem, we get
\begin{gather*}
\int_{\mathbb R}h_{B}(x)P_z(x)\,dx\simeq 2P_z(x_0)\left(Bh^{-1}(B)+\int_{h^{-1}(B)}^{\delta}h(t)\,dt\right).
\end{gather*}
Theorem \ref{thm: single deep zero} then gives
\begin{gather}
    \label{eq: asymp in four cases two terms}
    \frac{1}{P_z(x_0)}|\log m_r(z)| \simeq Bh^{-1}(B)+\int_{h^{-1}(B)}^{\delta}h(t)\,dt.
\end{gather}
In every case of the given function $h$ we estimate the asymptotic behaviour of the right-hand side explicitly, see Table \ref{table :examples of h}. The entries in the row corresponding to $h(x) = x^{-p}$ with $p > 1$ are exact while the entries in the other rows are given up to a multiplicative factor $1 + o(1)$ as $r\to\infty$. We notice that in the first two cases the integral term in \eqref{eq: asymp in four cases two terms} dominates, in the third case the terms are asymptotically the same and in the case (4) the first term prevails.

\begin{table}[ht]
\centering
\renewcommand{\arraystretch}{1.8}
\setlength{\tabcolsep}{10pt}
\resizebox{\textwidth}{!}{%
\begin{tabular}{ccccc}
\toprule
$h(x)$ & $B$ & $h^{-1}(B)$ & $Bh^{-1}(B)$ & $\displaystyle\int_{h^{-1}(B)}^{\delta}h(t)\,dt$ \\
\midrule
$\displaystyle \frac{1}{x\log(1/x)} $ & $\displaystyle \sqrt{\frac{2r}{\log r}}$ & $\displaystyle \sqrt{\frac{2}{r\log r}}$ & $\displaystyle \frac{2}{\log r}$ & $\displaystyle\log\log r$
\\[5mm]
$\displaystyle \frac{\log^p(1/x)}{x}$ & $\displaystyle  \sqrt{\frac{r \log^p r}{2^p}}$ & $\displaystyle \sqrt{\frac{\log^p r}{r2^p}}$ & $\displaystyle \frac{\log^pr}{2^p}$ & $\displaystyle \frac{(\log r)^{p+1}}{(p+1)2^{p+1}}$ 
\\[5mm]
$\displaystyle x^{-p}$ & $\displaystyle r^{p/(p+1)}$ & $\displaystyle r^{-1/(p+1)}$ & $\displaystyle r^{(p-1)/(p+1)}$ & $\displaystyle \frac{r^{(p-1)/(p+1)}}{p - 1}$
\\[2mm]
$\displaystyle\exp\left(x^{-p}\right)$ & $\displaystyle \frac{r}{(\log r)^{1/p}}$ & $\displaystyle \frac{1}{(\log r)^{1/p}}$ & $\displaystyle\frac{r}{(\log r)^{2/p}}$ & $\displaystyle\frac{r}{p(\log r)^{1+2/p}}$ \\[4mm]
\end{tabular}%
}
\caption{Asymptotic behaviour of $B$, $h^{-1}(B)$, and the two contributions to the integral in \eqref{eq: asymp in four cases two terms}.}
\label{table :examples of h}
\end{table}
\end{proof}
\end{corol}

\section{Proofs of auxiliary results}

\subsection{Embedding of Paley-Wiener spaces}\label{sect: paley-wiener embedding}
Throughout this section, let $r>0$ and let $\sigma$ be a positive Borel measure on $\mathbb R$. Recall that $\PW_{[0,r]}$ is the Paley--Wiener space of all functions $f\in L^2(\mathbb R)$ whose Fourier transforms are supported in $[0,r]$.

\begin{prop}[Theorem 1, Lemma 3 in \cite{lin1965equivalent}]
The inclusion $\PW_{[0,r]}\subset L^2(\mathbb R,\sigma)$ holds if and only if
\begin{gather*}
\sup_{x\in\mathbb R}\sigma([x,x+1])<\infty.
\end{gather*}
Moreover, whenever this condition is satisfied, the embedding $\PW_{[0,r]}\hookrightarrow L^2(\mathbb R,\sigma)$ is continuous.

\end{prop}

\begin{proof}
Suppose first that
\begin{gather*}
M:=\sup_{x\in\mathbb R}\sigma([x,x+1])<\infty.
\end{gather*}
Since every function in  $\PW_{[0,r]}$ is locally absolutely continuous, for each $n$ there exists $x_n\in[n,n+1]$ such that
\begin{gather*}
|f(x_n)|^2\le \int_n^{n+1}|f(t)|^2\,dt.
\end{gather*}
Hence, for every $x\in[n,n+1]$,
\begin{gather*}
|f(x)|^2
\le |f(x_n)|^2+2\int_n^{n+1}|f(t)f'(t)|\,dt
\le 2\int_n^{n+1}|f(t)|^2\,dt+\int_n^{n+1}|f'(t)|^2\,dt.
\end{gather*}
It follows that
\begin{gather*}
\int_{\mathbb R}|f(x)|^2\,d\sigma(x)
=\sum_{n\in\mathbb Z}\int_{n}^{n + 1}|f(x)|^2\,d\sigma(x)
\le M\left( 2\int_\R|f(x)|^2\,dx + \int_\R|f'(x)|^2\,dx\right).
\end{gather*}
The Plancherel theorem gives
\begin{gather*}
    \int_{\mathbb R}|f(x)|^2\,d\sigma(x) \le M(2 + r^2)\int_\R|f(x)|^2\,dx,
\end{gather*}
which implies the required inclusion.
\medskip

Conversely, assume that $\PW_{[0,r]}\subset L^2(\mathbb R,\sigma)$.
Consider the identity operator
\begin{gather*}
J:\PW_{[0,r]}\longrightarrow L^2(\mathbb R,\sigma),
\qquad Jf=f,
\end{gather*}
where $\PW_{[0,r]}$ is equipped with its usual $L^2(\mathbb R)$ norm. Operator $J$ is continuous by the closed graph theorem hence
there exists a constant $C>0$ such that
\begin{gather*}
\|f\|_{L^2(\R,\sigma)}
\le C\|f\|_{L^2(\mathbb R)},
\qquad f\in \PW_{[0,r]}.
\end{gather*}
Choose a function $\varphi\in \PW_{[0,r]}$ with  $\phi(0)\neq0$. For instance, we may take $\phi(x)=\frac{e^{irx} -1}{x}$. There exist $\delta>0$ and $c>0$ such that $|\phi(x)|\ge c$ for $|x|\le \delta$.
Now we write
\begin{gather*}
c^2\sigma([t-\delta,t+\delta])
\le \int_{\mathbb R}|\phi(x - t)|^2\,d\sigma(x)
\le C^2\|\phi\|_{L^2(\mathbb R)}^2.
\end{gather*}
To conclude we notice that every interval of length $1$ can be covered by the fixed number of intervals of length $\delta$.
\end{proof}

\begin{prop}
The space  $\PW_{[0,r]}$, equipped with the norm inherited from $L^2(\mathbb R,\sigma)$ 
is a reproducing kernel Hilbert space if and only if
\begin{gather*}
\int_{\mathbb R}|f(x)|^2\,dx
\simeq
\int_{\mathbb R}|f(x)|^2\,d\sigma(x),
\qquad f\in \PW_{[0,r]}.
\end{gather*}
\end{prop}

\begin{proof}
Suppose first that the two norms are equivalent, there exists a constant $C>0$ such that
\begin{gather}
\label{eq: double norm ineq in prop}
\frac{1}{C}\|f\|_{L^2(\mathbb R)}
\le \|f\|_{L^2(\R,\sigma)}
\le C\|f\|_{L^2(\mathbb R)},
\qquad f\in \PW_{[0,r]}.
\end{gather}
Since  $\PW_{[0,r]}$ is complete in its usual $L^2(\mathbb R)$ norm, it is also complete in the $L^2(\R,\sigma)$ norm.
Moreover, for every $z\in\mathbb C$, the Fourier representation and the Cauchy--Schwarz inequality imply
\begin{gather*}
|f(z)|
\ls \sqrt{r}e^{|\Im z|r}\|f\|_{L^2(\mathbb R)}\le C\sqrt{r}e^{|\Im z|r}\|f\|_{L^2(\R,\sigma)},
\qquad f\in \PW_{[0,r]}.
\end{gather*}
Thus, every point evaluation is continuous with respect to the $L^2(\R,\sigma)$ norm, and hence  $\PW_{[0,r]}$ is a reproducing kernel Hilbert space.

Conversely, suppose that  $\PW_{[0,r]}$, equipped with the $L^2(\R,\sigma)$ norm, is a reproducing kernel Hilbert space. In particular, it is complete, and every function in  $\PW_{[0,r]}$ belongs to $L^2(\R,\sigma)$. By the preceding proposition, the identity operator 
\begin{gather*}
J:\bigl(\PW_{[0,r]},\|\cdot\|_{L^2(\mathbb R)}\bigr)
\longrightarrow
\bigl(\PW_{[0,r]},\|\cdot\|_{L^2(\R,\sigma)}\bigr)
\end{gather*}
is continuous, i.e., the right inequality in \eqref{eq: double norm ineq in prop} holds with some $C > 0$. By the open mapping theorem $J^{-1}$ is also continuous and the left inequality in \eqref{eq: double norm ineq in prop} is also satisfied.
\end{proof}

\subsection{Estimates of convolutions with Poisson kernel}\label{sec: proof of lemmas} 

\begin{proof}[Proof of Lemma \ref{lem: phi poisson}]
Before proceeding to the proof notice that the assertion $\phi(x/2)\ls \phi(x)$ from \eqref{Reg2} also holds under the assumption (i):
\begin{gather*}
    (x/2)^2\phi(x/2)\le x^2\phi(x), \qquad \phi(x/2)\le 4 \phi(x). 
\end{gather*}
Since both $\tilde\phi$ and $P_{ih}$ are even it suffices to consider only $x\ge 0$. 
If $B\le \phi(x/2)$, then the lemma is trivial. Indeed, we have
\begin{gather*}
(P_{ih}*\widetilde \phi_B)(x)
\le
\max_{\mathbb R}\widetilde \phi_B
=
B
\le
\phi(x/2)
\ls
\phi(x).
\end{gather*}
Thus we may assume $B\ge \phi(x/2)$ or, equivalently, $x\ge 2\phi^{-1}(B)$. We claim that in this situation $h\ls x$ holds: condition $(i)$ gives $h\le \phi^{-1}(B)\le x$ and condition $(ii)$ implies $h\ls\frac1B\ls \phi^{-1}(B)\le x$.
\smallskip

We rewrite the convolution as
\begin{align}
    \label{eq: convolution as two integrals}( \widetilde\phi_B*P_{ih})(x) =
    \int_{|y|\le x/2}\widetilde\phi_B(y)P_{ih}(x-y)\,dy
    +
    \int_{|y|\ge x/2}\widetilde\phi_B(y)P_{ih}(x-y)\,dy.
\end{align}
The function $ \widetilde\phi_B$ is monotone on $\mathbb R_+$, hence we have
\begin{gather}
\label{eq: first int}
\int_{|y|\ge x/2} \widetilde\phi_B(y)P_{ih}(x-y)\,dy
\le
 \widetilde\phi_B(x/2) \int_{\mathbb R}P_{ih}(x-y)\,dy
=
 \phi_B(x/2) \ls \phi(x).
\end{gather}
For the first term in \eqref{eq: convolution as two integrals} we use the estimate 
\begin{gather}
    \label{eq: estmate for Poisson kernel}
    P_{ih}(t) = \frac{h}{\pi|t - ih|^2}\le \frac{h}{\pi t^2}
\end{gather}
which gives $P_{ih}(x - y)\ls \frac{h}{(x - y)^2} \ls \frac{h}{x^2}$ when $|y|\le x/2$.
We get the inequality
\begin{gather}
\label{eq: estimate integral phi B}
\int_{|y|\le x/2}
 \widetilde \phi_B(y)P_{ih}(x-y)\,dy
\ls
\frac{h}{x^2}
\int_0^{x/2}
\phi_B(y)\,dy.
\end{gather}
Assume that condition $(i)$ holds. Then we have
\begin{gather*}
\int_0^{x/2}\phi_B(y)\,dy
=
\int_0^{\phi^{-1}(B)}\phi_B(y)\,dy
+
\int_{\phi^{-1}(B)}^{x/2}\phi_B(y)\,dy = B\phi^{-1}(B) + \int_{\phi^{-1}(B)}^{x/2}\phi(y)\,dy.
\end{gather*}
Assertion $(i)$ implies $\phi(y)\ls x^2\phi(x) / y^2$ for $y\le x$ hence
\begin{gather*}
\int_0^{x/2}\phi_B(y)\,dy
\le
B\phi^{-1}(B)
+
\int_{\phi^{-1}(B)}^{x/2}
\frac{x^2\phi(x)}{y^2}\,dy 
\le 
B\phi^{-1}(B)
+
\frac{x^2\phi(x)}{\phi^{-1}(B)}.
\end{gather*}
We also have $\phi^{-1}(B)\le x$ and from $(i)$ we get $B\phi^{-1}(B)\le \frac{x^2\phi(x)}{\phi^{-1}(B)}$. It follows that
\begin{gather*}
\int_0^{x/2}\phi_B(y)\,dy
\ls
\frac{x^2\phi(x)}{\phi^{-1}(B)}.
\end{gather*}
The latter, \eqref{eq: convolution as two integrals}, \eqref{eq: first int} and \eqref{eq: estimate integral phi B} give
\begin{gather*}
(\widetilde \phi_B*P_{ih})(x)
\le
\phi(x)
+
\frac{h}{x^2}
\cdot
\frac{x^2\phi(x)}{\phi^{-1}(B)}
=
\phi(x)\left( 1+
\frac{h }{\phi^{-1}(B)}
\right) \ls \phi(x),
\end{gather*}
where the last inequality follows from the second part of $(i)$.
\smallskip

Now assume that $\phi$ satisfies $(ii)$. The integral in \eqref{eq: estimate integral phi B} can be bounded as
\begin{gather*}
\int_0^{x/2}\widetilde \phi_B(y)\,dy
\ls Bx.
\end{gather*}
Again this with \eqref{eq: convolution as two integrals}, \eqref{eq: first int} and \eqref{eq: estimate integral phi B} give
\begin{gather*}
(P_{ih}*\widetilde \phi_B)(x)
\ls
\phi(x)
+
\frac{h}{x^2}Bx
=
\phi(x)
+
\frac{h B}{x}.
\end{gather*}
To conclude the proof we recall that $h B\ls 1$ and $1/x\ls \phi(x)$ by $(ii)$.
\end{proof}

\begin{proof}[Proof of Lemma \ref{lemma: poisson kernel of subordinated function}]
Fix $x\in\mathbb R$ and let $y_x=\phi^{-1}(H(x))$.
Similarly to the previous lemma, the claim is trivial whenever $\phi^{-1}(B)\ge y_x/2$.
Indeed, in this case we have
\begin{gather*}
    (H_B*P_{ih})(x) \le \max_{\R} H_B = B \le \phi(y_x/2) \ls \phi(y_x) = H(x).
\end{gather*}
Thus below we may assume $\phi^{-1}(B)\le y_x/2$ or, equivalently $B\ge\phi(y_x/2)$. Rewrite the convolution as
\begin{gather}
\label{eq: conv as two integrals}
(H_B*P_{ih})(x)
=
\int_{|y|>\frac{y_x}{2}} H_B(x-y)P_{ih}(y)\,dy
+
\int_{|y|\le \frac{y_x}{2}} H_B(x-y)P_{ih}(y)\,dy.
\end{gather}
For the first integral we use the simple estimate
$H_B\le B$ to obtain
\begin{gather*}
\int_{|y|>\frac{y_x}{2}}
H_B(x-y)P_{ih}(y)\,dy
\le
2B
\int_{\frac{y_x}{2}}^\infty
P_{ih}(y)\,dy.
\end{gather*}
Now recall \eqref{eq: estmate for Poisson kernel} the bound $P_{ih}(y)\le h / y^2$ for the Poisson kernel. It gives
\begin{gather*}
2B\int_{\frac{y_x}{2}}^\infty P_{ih}(y)\,dy
\le
2B
\int_{\frac{y_x}{2}}^\infty
\frac{h}{y^2}\,dy
\ls \frac{Bh}{y_x}.
\end{gather*}
If $\phi$ satisfies \eqref{Reg2}, then
\begin{gather*}
\frac{Bh}{y_x}
\ls
\frac1{y_x}
\ls
\phi(y_x).
\end{gather*}
If \eqref{Reg1} holds, then $h\le\phi^{-1}(B)$ and inequality $\phi^{-1}(B) \le y_x/2$ implies $B\phi^{-1}(B)\le y_x/2\,\phi(y_x/2)$. We get
\begin{gather*}
\frac{Bh}{y_x}
\le
\frac{B\phi^{-1}(B)}{y_x}
\le
\frac{y_x/2\,\phi(y_x/2)}{y_x}
=
\frac12\phi(y_x/2)
\ls
\phi(y_x).
\end{gather*}
In both cases the first integral in \eqref{eq: conv as two integrals} does not exceed $\phi(y_x)=H(x)$.
\smallskip 

For the second integral in \eqref{eq: conv as two integrals} we use the subordination to $\phi$.
The Lipschitzness of the function $\phi^{-1}\circ H$  gives
\begin{gather*}
\phi^{-1}(H_B(x-y)) \ge \phi^{-1}(H(x-y)) \ge \phi^{-1}(H(x)) -|y| = y_x - |y|.
\end{gather*}
When $|y| < y_x$ the value on the right is positive hence
\begin{gather*}
H_B(x-y)\le \phi_B(y_x-|y|).
\end{gather*}
This shows that the second integral does not exceed
\begin{gather*}
\int_{|y|\le \frac{y_x}{2}} H_B(x-y)P_{ih}(y)\,dy
\le
\int_{|y|\le \frac{y_x}{2}}
\phi_B(y_x-|y|)P_{ih}(y)\,dy
\le
2(\widetilde \phi_B*P_{ih})(y_x).
\end{gather*}
Lemma \ref{lem: phi poisson} finishes the proof.
\end{proof}

\bibliographystyle{plain} 
\bibliography{ref}

\end{document}